\documentclass[12pt]{article}

\usepackage[dvipsnames]{xcolor}

\usepackage{amssymb,mathrsfs,amsmath}
\usepackage{amscd,amsmath,amsthm,amssymb,amsfonts,enumerate}
\usepackage{graphicx}

\usepackage{subfigure}

\usepackage{hyperref}
\hypersetup{
	colorlinks=true,
	linkcolor=cyan,
	filecolor=magenta,
	urlcolor=cyan,
	citecolor=cyan,
}

\def\({\left(}
\def\]{\right]}
\def\[{\left[}
\def\){\right)}

\newtheorem{thm}{Theorem}[section]

\newtheorem{lem}[thm]{Lemma}
\newtheorem{rem}[thm]{Remark}

\newtheorem{defn}[thm]{Definition}
\newtheorem{exm}[thm]{Example}

\def\P{{\mathbb P}}
\def\E{{\mathbb E}}

\newcommand{\disp}{\displaystyle}
\newcommand{\bea}{$$\begin{array}{ll}}
	\newcommand{\eea}{\end{array}$$}
\newcommand{\bed}{\begin{displaymath}}
	\newcommand{\eed}{\end{displaymath}}
\newcommand{\ad}{&\!\!\!\disp}
\newcommand{\aad}{&\disp}
\newcommand{\barray}{\begin{array}{ll}}
	\newcommand{\earray}{\end{array}}
\newcommand{\beq}[1]{\begin{equation} \label{#1}}
	\newcommand{\eeq}{\end{equation}}

\newcommand{\bedd}{\bed\begin{array}{l}}
	\newcommand{\eedd}{\end{array}\eed}
\newcommand{\al}{\alpha}

\newcommand{\e}{\varepsilon}

\newcommand{\nd}{\noindent}

\newcommand{\M}{{\cal{M}}}
\newcommand{\one}{{1}}

\newcommand{\wdt}{\widetilde}

\newcommand{\R}{\mathbb R}
\newcommand{\lbar}{\overline}

\def\one{{\hbox{1{\kern -0.35em}1}}}

\def\ka{\kappa}

\def\({\left(}
\def\]{\right]}
\def\[{\left[}
\def\){\right)}
\def\one{{\hbox{1{\kern -0.35em}1}}}

\makeatletter \@addtoreset{equation}{section}

\def\N{\mathbb N}

\usepackage{lineno}

\begin{document}

\title{Mean-square exponential stability of exact and numerical solutions for neutral stochastic delay differential equations with Markovian switching}

\author{
Jina Yang\thanks{ Department of Biostatistics, Brown University, 121 South Main Street, 
Providence, Rhode Island, USA 02912, jina\_yang@brown.edu  } \and 
Ky Quan Tran\thanks{Corresponding author, Department
	of Applied Mathematics and Statistics,
	State University of New York - Korea, Yeonsu-Gu, Incheon, Korea 21985, ky.tran@stonybrook.edu } 
}

\date{}

\maketitle

\begin{abstract}
This paper investigates the mean-square exponential stability of neutral stochastic delay differential  equations (NSDDEs) with Markovian switching. The analysis addresses the complexities arising from the interaction between the neutral term, time-varying delays, and structural changes governed by a continuous-time Markov chain. We establish novel and practical criteria for the mean-square exponential stability of both the underlying system and its numerical approximations via the Euler-Maruyama method. Furthermore, we prove that the numerical scheme can reproduce the exponential decay rate of the true solution with arbitrary accuracy, provided the step size is sufficiently small. The theoretical results are supported by a numerical example that illustrates their effectiveness.

\vskip 0.3 true in \nd{\bf Key words.} neutral stochastic delay differential equations; mean-square exponential stability; time-dependent delay; Markovian switching; Euler-Maruyama approximations

\vskip 0.3 true in \nd{\bf Mathematics subject classification.}  60H10, 34K50, 65C30

\end{abstract}

\bigskip

\bigskip

\newpage

\setlength{\baselineskip}{0.25in}

\section{Introduction}\label{sec:int}

This work aims to investigate the mean-square exponential stability of a class of neutral stochastic delay differential equations (NSDDEs) with Markovian switching. Our motivation stems from the need to model complex real-world systems where the current state depends on past events. Such stochastic processes, which incorporate a neutral term, a time-varying delay, and white noise, are widely applicable in many disciplines; see \cite{Mao2007,Zhu} and the references therein.

It has been recognized that stochastic processes incorporating time-dependent delays are increasingly vital in fields ranging from biology to engineering \cite{K2013,N2003}. Unlike constant delays, time-dependent delays capture the dynamic nature of these systems. Furthermore, as presented in \cite{Mao2006} and \cite{Zhu}, the inclusion of Markovian switching is essential to describe random environmental influences. What happens if the environment changes randomly? The asymptotic behavior becomes significantly more complex than in a fixed environment. In particular, there are examples (see \cite{Zong2014}) where a system that is stable in some fixed environments and unstable in others can be stabilized or destabilized purely by choosing suitable switching rates.

The literature on the exponential stability of NSDDEs is rich. By and large, intensive studies of NSDDEs without switching can be found in \cite[Chapter 6]{Mao2007}. There have also been efforts to find criteria for the mean-square exponential stability of NSDDEs with Markovian switching. However, existing works such as \cite{Yuan2015,M2017S,Mao2021,Zhang2019} focus primarily on cases where the delay is either constant or differentiable. Meanwhile, although \cite{Mao2023} treats a class of NSDDEs with non-differentiable time delays, the delay functions are assumed to satisfy certain additional conditions. To the best of our knowledge, there are no explicit criteria for the mean-square exponential stability of NSDDEs with Markovian switching subject to arbitrary time-dependent delays. Thus, a natural question arises: Can we develop an explicit criterion for the mean-square exponential stability of NSDDEs with arbitrary time-dependent delays?

To answer this question, we seek a criterion given in terms of the system coefficients and the generator of the Markovian switching, without posing specific technical conditions on the time-dependent delay functions. This objective motivates the first part of our current paper.

Because of the stochasticity, nonlinearity, and hybrid switching behavior, closed-form solutions for these systems are virtually impossible to obtain. Hence, numerical methods are often the only viable alternative. A wide range of numerical methods have been developed to solve NSDDEs, with considerable attention given to preserving their exponential stability \cite{M2017a,Mao2008,Zong2015,Zhou2013}. For instance, the exponential stability of the Euler–Maruyama (EM) method has been studied in \cite{M2013,Ky3,Zong2016,Zhou2013}. The backward Euler-Maruyama and theta Euler-Maruyama approximations have also been investigated in \cite{Yuan2015,M2014,Mo2017c,Mo2021,Mo2017b,Ky3,Zong2015}, along with studies examining highly nonlinear NSDDEs under specific scenarios \cite{M2013,M2017a,Zhou2015}.

Focusing on the long-time behavior of the numerical solutions, a question naturally arises: Given a stable equation, can the numerical solutions reproduce the stability? This question motivates the second part of our paper. It should be noted that a key limitation in \cite{Yuan2015,M2017S,Zhang2019} is that the time-dependent delay $\delta(t)$ is assumed to be either a constant function or differentiable with $\sup_{t \ge 0} \delta'(t) < 1$. This assumption is also prevalent in other existing work \cite{Yuan2015,Zhu16,Mo2017c,Mo2021,Mo2017b,Zong2015}.

In this paper, our effort is placed on overcoming these limitations on the time-dependent delay $\delta(t)$ by using a novel comparison approach via contradiction. While our approach is motivated by developments in \cite{Ky3,Ky2}, the stability analysis herein is considerably more challenging due to the need to treat Markovian switching.

The rest of the paper is arranged as follows. Section \ref{sec:for} begins with the problem formulation. Section \ref{sec:exact} presents a new criterion for the mean-square exponential stability of the exact solutions. Section \ref{sec:num} focuses on the stability of the EM numerical solutions for NSDDEs. Moreover, an example is presented for illustrative purposes. Finally, Section \ref{sec:con} concludes the paper with a few more remarks.

\paragraph{Notation.} Let $\mathbb{R}_+$ be the set of nonnegative real numbers, $\mathbb{N}$ be the set of positive integers, and $\mathbb{N}_0=\mathbb{N}\cup \{0\}$.  
For two real numbers $a$ and $b$, $a\vee b$ denotes $\max\{a, b\}$. For a real number $u$,
$\lfloor u \rfloor$ is the integer part (or the floor) of $u$.
 For a matrix $A\in \mathbb{R}^{d_1\times d_2}$, $A^\top$ denotes its transpose and $|A| = \sqrt{\rm trace(A^\top A)}$ is its trace norm. 
 For $x=(x_1, \dots, x_d)^\top\in \mathbb{R}^d$, its Euclidean norm is denoted by
$|x|=\big(\sum_{i=1}^d x_i^2\big)^{1/2}$.  For a set $M$, the indicator function of $M$ is denoted by ${\bf 1}_M$.
  For a constant $r>0$, let $C([-r, 0], \mathbb{R}^d)$  be the Banach space of all continuous functions on $[-r, 0]$, endowed with the norm $\|\phi\|= \max_{s\in [-r, 0]} |\phi(s)|$ for $\phi\in C([-r, 0], \mathbb{R}^d)$.  For a $\sigma$-algebra $\mathcal{F}_0$,
let $C^{b}_{\mathcal{F}_0}([-r, 0], \mathbb{R}^d)$ be the family of $\mathcal{F}_0$-measurable bounded $C([-r, 0], \mathbb{R}^d)$-valued random variables.

\section{Formulation}\label{sec:for}

We begin by specifying the probabilistic framework. To facilitate the analysis, we work with a complete probability space $(\Omega, \mathcal{F}, \P, \{\mathcal{F}_t\})$, where the filtration $\{\mathcal{F}_t\}$ satisfies the usual conditions (i.e., it is right-continuous and ${\mathcal F}_0$ contains all $\P$-null sets). Let $w(t)$ be an $m$-dimensional standard Brownian motion. To model the regime switching, we introduce a continuous-time Markov chain $\alpha(t)$ taking values in a finite state space $\M=\{1, 2, \dots, m_0\}$ (where $m_0\in \N$) with generator $Q=(q_{ij})\in \R^{m_0\times m_0}$. We assume that $w(t)$ and $\alpha(t)$ are independent and adapted to $\{\mathcal F_t\}$.

The evolution of the switching process $\alpha(t)$ is governed by the transition probability specification:
\beq{e.main3}
\P\Big(\al(t+\Delta t)=j|\al(t)=i \Big)
= \begin{cases} q_{ij}\Delta t + o(\Delta t) &\ \  \text{if}
\ \  i\not =j,  \\
1 + q_{ii}\Delta t + o(\Delta t) &\ \  \text{if}  \ \  i=j,
\end{cases}
\eeq
where $q_{ij}\ge 0$ for $i\ne j$ and the row sums satisfy $\sum_{j\in \M} q_{ij}=0$ for any $i\in \M$.

Our main interest lies in the dynamics of a  NSDDE. Let $r>0$ be the delay bound. We consider Borel measurable functions $f: \R^d\times \R^d\times \M\to \R^d$, $g: \R^d\times \R^d\times \M\to \R^{d\times m}$, $G: \R^d\times \M\to \R^d$, and a time-dependent delay function $\delta: \R_+\to [0, r]$. The state trajectory $X(t)$ is determined by the following system:
\beq{e.main}
\barray
d \big( X(t) - G(X(t-\delta(t), \al(t)) \big)\ad  = f \big( X(t), X(t-\delta(t)),  \al(t)\big) dt \\
\ad \ \  + g\big( X(t), X(t-\delta(t)),  \al(t) \big)dw(t)  \ \ \text{for} \ \ t\ge 0.
\earray
\eeq
To ensure the system is well-defined, we impose the initial condition $(\xi, i_0)\in C^{b}_{\mathcal{F}_{0}}([-r, 0], \R^d)\times \M$, such that
\beq{e.main2}
X(s) = \xi(s) \ \  \text{for any} \ \  s\in [-r, 0], \ \ \al(0)=i_0.
\eeq

To proceed with the analysis, we require standard regularity conditions on the coefficients. Throughout the paper, we assume that the functions $f(x, y, i)$ and $g(x, y, i)$ satisfy a local Lipschitz condition. That is, for each $n\in \mathbb{N}$, there exists a constant $\wdt K_{n}>0$ such that
\bea
 |f(x, y, i)-f(\lbar x, \lbar y, i)|
\ad +|g(x, y, i)-g(\lbar x, \lbar y, i)|  \le \wdt K_{n}\big( |x-\lbar x|+|y-\lbar y| \big)
\eea
whenever $|x|\vee |y|\le n$, $|\lbar x|\vee |\lbar y|\le n$, and $i\in \M$. Furthermore, we assume $f(0, 0, i)=0\in \R^d$ and $g(0, 0, i)=0\in \R^{d\times m}$ for any $i\in \M$. Regarding the neutral term, we assume $G(0, i)=0\in \R^d$ and that there exists a contraction constant $\ka\in (0, 1)$ satisfying
$$
|G(x, i)-G(y, i)|\le \ka |x-y| \ \ \text{for}  \ \ x, y\in \R^d, \ \ i\in \M.
$$

As a standing assumption, we posit that for each initial condition $(\xi, i_0) \in C^{b}_{\mathcal{F}_0}([-r, 0], \R^d)\times \M$, Eq. \eqref{e.main} admits a unique global solution $X^{\xi, i_0}(\cdot)$ satisfying \eqref{e.main2}. It has been recognized that if $f(\cdot)$ and $g(\cdot)$ satisfy the linear growth condition; that is, there exists $\wdt K_{0}>0$ such that
$$ |f(x, y, i)|
+|g(x, y, i)|  \le \wdt K_{0}\big(1+ |x|+|y| \big) \ \ \text{for} \ \ x, y\in \R^d, \ \ i\in \M,
$$
then the existence and uniqueness of the global solution are guaranteed; see \cite[Theorem 2.2, p. 204]{Mao2007}.

Our ultimate goal is to analyze the stability of these solutions. We adopt the following definition of mean-square exponential stability.

\begin{defn} \
 {\rm
 Eq. \eqref{e.main} is said to be mean-square exponentially stable if there exist constants $K>0$ and $\lambda>0$ such that
$$\E |X^{\xi, i_0}(t)|^2 \le K e^{-\lambda t} \E\|\xi\|^2 \ \  \text{for any}\ \ (\xi, i_0) \in C^{b}_{\mathcal{F}_0}([-r, 0], \R^d)\times \M \ \ \text{and} \ \ t\ge 0.
$$
}
\end{defn}

\section{Stability of the exact solutions} \label{sec:exact}

To proceed with the stability analysis, we require the following assumption.

\begin{itemize}
\item[\rm(H)] There are positive constants $\{p_i\}_{i\in \M}$ 
and nonnegative constants $c_0$, $c_1$, and $c_2$ 
such that
\beq{e.cond}
\barray
\ad 2 \big(x-G(y, i)\big)^\top f(x, y, i)  + |g(x, y, i)|^2 + \sum\limits_{j\in \M} q_{ij} \dfrac{p_j}{p_i} |x-G(y, i)|^2\\
\ad \qquad \le -c_0  |x-G(y, i)|^2 + c_1 |x|^2 + c_2  |y|^2
\earray
\eeq
for any $(x, y, i)\in \R^d\times \R^d\times \M$. 

\end{itemize}

For the sake of notational simplicity, we define
$$
p_{\max} = \max\{p_i: i\in \M\},  \ \ p_{\min} = \min\{p_i: i\in \M\},  \ \ r_p = \dfrac{p_{\max}}{p_{\min}},  
$$
and
$$Z(t)=X(t) - G(X(t-\delta(t), \al(t)) \ \ \text{for} \ \ t\ge 0.$$
To highlight the dependence on the initial condition \eqref{e.main2}, the process $Z(t)$  
 is denoted by $Z^{\xi, i_0}(t)$.

\begin{lem} \label{lem:1}
Assume that there exist constants $T>0$, $K>0$, and $\lambda>0$ such that 
$$\E\big[ p_{\al(t)} |Z^{\xi, i_0}(t)|^2 \big] \le K e^{ - \lambda t} \E \|\xi\|^2 \ \ \text{for any} \ \  (\xi, i_0, t)\in C^{b}_{\mathcal{F}_0}([-r, 0], \R^d) \times \M\times [0, T).$$
Then
$$
 \E \big[ p_{\al(t)}  |X^{\xi, i_0}(t)|^2 \big] \le \dfrac{r_p}{(1-\ka  e^{\lambda r/2} )^2} K e^{-\lambda t} \E \|\xi\|^2
$$
and
$$
 \E \big[ p_{\al(t)}  |X^{\xi, i_0}(t-\delta(t))|^2 \big] \le \dfrac{r_p}{(1-\ka  e^{\lambda r/2} )^2} K e^{-\lambda (t-\delta(t))} \E \|\xi\|^2
$$
for any  $(\xi, i_0, t)\in C^{b}_{\mathcal{F}_0}([-r, 0], \R^d)\times \M\times [0, T).$
\end{lem}

  \begin{proof}
To proceed, we recall the H\"older inequality. For any $\e>0$,
$$|x+y|^2 \le (1+\e) \Big(|x|^2 +  {|y|^2 \over \e}\Big) \ \ \text{for} \ \ x, y\in \R^d;$$
see \cite[Lemma 4.1, p. 211]{Mao2007}.
By setting $v\in (0, 1)$ and $\e = \frac{v}{1-v}$, we obtain
\beq{b.6}
|x+y|^2 \le \dfrac{1}{1-v}|x|^2 +  \dfrac{1}{v}|y|^2 \ \ \text{for} \ \  x, y\in \R^d.
\eeq
For the sake of notational simplicity, we write $X(t) = X^{\xi, i_0}(t)$ and $Z(t)=X(t)-G(X(t-\delta(t)), \al(t))$. 
Let $K_1 =K \E\|\xi\|^2$.
In view of \eqref{b.6}, for any $t\in [0, T)$, it follows that
\bea
\ad \big|Z(t) + G(X(t-\delta(t)), \al(t))\big|^2 \le  \dfrac{1}{1-v} |Z(t)|^2 + \dfrac{1}{v} |G(X(t-\delta(t)), \al(t))|^2.
\eea
It follows that
$$|X(t)|^2\le \dfrac{1}{1-v}|Z(t)|^2+ \dfrac{\ka^2}{v}  |X(t-\delta(t))|^{2}.$$
Multiplying by $e^{\lambda t}$ and taking expectations, we arrive at
\bea
e^{\lambda t} \E|X(t)|^2\ad \le \dfrac{1}{1-v} e^{\lambda t} \E|Z(t)|^2  +  \dfrac{ \ka^2}{v} e^{\lambda t}\E|X(t-\delta(t))|^2\\
\ad \le \dfrac{1}{1-v} e^{\lambda t}\E |Z(t)|^2  +  \dfrac{ \ka^2}{v} e^{-\lambda (t-\delta(t))+\lambda t} e^{\lambda (t-\delta(t))}\E | X(t-\delta(t))|^2 \\
\ad \le \dfrac{1}{1-v} e^{\lambda t}\E |Z(t)|^2  +  \dfrac{ \ka^2}{v}e^{\lambda \delta(t)} \Big(\sup\limits_{s\in [0, T)}e^{\lambda s}\E|X(s)|^2 \Big)\\
\ad \le  \dfrac{1}{1-v} e^{\lambda t}\E |Z(t)|^2 +
\dfrac{ \ka^2}{v}e^{\lambda r} \Big(\sup\limits_{s\in [0, T)}e^{\lambda s}\E |X(s)|^2 \Big);
\eea
which leads to
\bea
e^{\lambda t} \E |X(t)|^2  \le \dfrac{1}{1-v } e^{\lambda t}\E |Z(t)|^2 +   \dfrac{\ka^2}{v} e^{\lambda r}\Big(\sup\limits_{s\in [0, T)}e^{\lambda s}\E |X(s)|^2  \Big).
\eea
Setting $v = \kappa e^{\lambda r/2}$ and noting that
$$
e^{\lambda t}\E|Z(t)|^2
\le \frac{1}{p_{\min}} e^{\lambda t}\E\big[ p_{\alpha(t)} |Z(t)|^2 \big]
\le \frac{1}{p_{\min}} K_1
$$
yields
\beq{e.KJ}
e^{\lambda t}\E  |X(t)|^2   \le \dfrac{K_1}{( 1-\ka e^{\lambda r/2} )p_{\min} }   +\ka e^{\lambda r/2}\Big(\sup\limits_{s\in [0, T)}e^{\lambda s} \E |X(s)|^2\Big)
\ \ \text{for} \ \ t\in [0, T).
\eeq
Taking the supremum over $t\in [0, T)$ in \eqref{e.KJ} yields
$$ \sup\limits_{t\in [0, T)}e^{\lambda t} \E |X(t)|^2 \le \dfrac{ K_1 }{\big( 1-\ka e^{\lambda r/2}  \big)^2 p_{\min}}.$$
This implies
\beq{e.jk} 
\E |X(t)|^2 \le \dfrac{ K_1 }{\big( 1-\ka e^{\lambda r/2}  \big)^2 p_{\min}} e^{-\lambda t}\ \ \text{for} \ \ t\in [0, T).
\eeq
Consequently,
$$\E\big[ p_{\al(t)} |X(t)|^2 \big]\le p_{\max} \E |X(t)|^2  \le \dfrac{ K_1 p_{\max} }{\big( 1-\ka e^{\lambda r/2}  \big)^2 p_{\min} } e^{-\lambda t}\ \ \text{for} \ \ t\in [0, T);$$
that is,
$$\E\big[ p_{\al(t)} |X(t)|^2 \big] \le \dfrac{r_p K_1}{(1-\ka e^{\lambda r/2})^2}e^{-\lambda t} \ \ \text{for} \ \ t\in [0, T),$$
where $r_p = \dfrac{p_{\max}}{p_{\min}}$. 

For $t\in [0, T)$ with $t-\delta(t)\ge 0$, in view of \eqref{e.jk}, we have
\bea
 \E\big[ p_{\al(t)} |X(t-\delta(t))|^2\big]\ad  \le p_{\max} \E |X(t-\delta(t))|^2  \\
\ad \le 
\dfrac{ r_p K_1}{  (1-\ka e^{-\lambda r/2})^2}  e^{-\lambda (t-\delta(t))}.
\eea
Note that the inequality remains valid if $t-\delta(t)<0.$ Consequently,
$$\E\big[ p_{\al(t)} |X(t-\delta(t))|^2\big] \le \dfrac{r_p K_1}{  (1-\ka e^{-\lambda r/2})^2 }  e^{-\lambda (t-\delta(t))} \ \ \text{for} \ \ t\in [0, T).$$
The conclusion follows.
\end{proof}

 We are now in a position to present a criterion for the mean-square exponential stability of Eq. \eqref{e.main}.

\begin{thm} \label{thm:1}
Suppose that assumption {\rm (H)} is satisfied. In addition, assume that
\beq{ee.3}
-c_0 +\dfrac{r_p}{(1-\ka)^2}(c_1+ c_2)<0.
\eeq
Then Eq. \eqref{e.main} is mean-square exponentially stable. Specifically, for any $(\xi, i_0)\in C^{b}_{\mathcal{F}_0}([-r, 0], \R^d)\times \M$, $\E |X^{\xi, i_0}(t)|^2$ exponentially decays
with a rate no less than $\lambda_0$, which is the unique positive root in $(0, -\frac{2}{r}\ln \ka)$ of 
\beq{ee.4}\notag
\barray
\ad -c_0+  \dfrac{r_p}{ (1-\ka  e^{\lambda_0 r/2})^2}(c_1+   c_2 e^{\lambda_0 r}) + \lambda_0=0.
\earray
\eeq
\end{thm}

 \begin{proof}
Let $(\xi, i_0) \in C^{b}_{\mathcal{F}_0}([-r, 0], \R^d)\times \M$. Without loss of generality, we assume $\E\|\xi\|^2>0$. 
For the sake of notational simplicity, we write $X(t) = X^{\xi, i_0}(t)$ and $Z(t)=X(t)-G(X(t-\delta(t)), \al(t))$ for any $t\ge 0$. 

  Define 
$$H(\rho) = -c_0+   \dfrac{r_p}{(1-\ka e^{\rho r/2})^{2}} \big(c_1 + c_2 e^{\rho r}\big) + \rho\ \ \text{for} \ \ \rho\in (0, -\frac{2}{r}\ln \ka).$$
Observe that the function $H(\cdot)$ is continuous and strictly increasing on $(0,  -\frac{2}{r}\ln \ka)$.  In view of \eqref{ee.3}, $\lim_{\rho \to 0}H(\rho)<0$. Moreover, it is evident that  $\lim_{\rho \to -\frac{2}{r}\ln \ka} H(\rho)=\infty$. Then on the interval $(0,  -\frac{2}{r}\ln \ka)$,
the equation $H(\rho)=0$ has a unique solution $\rho=\lambda_0$.  Let $\lambda \in (0, \lambda_0]$ and $\beta=(1-\ka e^{\lambda r/2})^{-2}$. Since $H(\cdot)$ is strictly increasing on $(0,  -\frac{2}{r}\ln \ka)$, we have  $H(\lambda)\le H(\lambda_0)=0$. Consequently,
\beq{ee.3b}
- c_0 +  r_p\beta( c_1  +c_2e^{\lambda r} ) \le -\lambda.
\eeq
 Let $K=2(1+\ka^2)\sum_{i\in \M}p_i$ and $K_1= K\E \|\xi\|^2$. Then
\beq{Phi0}\E\big[p_{\al(t)} |Z(t)|^2 \big] < K_1 e^{-\lambda t} \ \ \text{when}\ \  t=0.\eeq
We aim to show that
\beq{ee.8d} \E \big[ p_{\al(t)} |Z(t)|^2\big] < K_1 e^{-\lambda t} \ \ \text{for any} \ \ t\ge 0.\eeq
 Let us consider the real-valued functions $\Phi(t)$ and $\Psi(t)$ defined  by
$$
\Phi(t) = \E \big[ p_{\al(t)} |Z(t)|^2 \big], \ \ 
\Psi(t)= K_1  e^{-\lambda t} \ \ \text{for}\ \ t\ge 0.
$$
Noting that the trajectories of $Z(t)$ are continuous, while the trajectories of $p_{\al(t)}$ are right-continuous, we see that $\Phi(t)$ is right-continuous on the interval $[0, \infty)$. Furthermore, it is clear that $\Psi(t)$ is continuous on $[0, \infty)$. 
Were the statement \eqref{ee.8d} false, by the right-continuity of $\Phi(t)$ and the continuity of $\Psi(t)$,
there would exist $t_*>0$
such that
\beq{ee.9d}\Phi(t)<\Psi(t) \ \ \text{for} \ \  t\in [0, t_*), \ \  \Phi(t_*)\ge  \Psi (t_*).\eeq
Fix a constant $c > c_0$.
Applying the It\^{o} formula to the process $Z(t)$ given by Eq.  \eqref{e.main} and taking expectation of both sides of the resulting equation, we derive
\beq{ee.10e}
\barray
e^{c t_*} \E \big[ p_{\al(t_*)} |Z(t_*)|^2 \big]\ad = \E\big[p_{i_0} |Z(0)|^2\big] + \disp\int_{0}^{t_*} e^{c s}\E \big[{\cal H}(s) \big]ds,
\earray
\eeq
where
\bea
{\cal H}(s)\ad =  c p_{\al(s)} |Z(s)|^2+ \sum\limits_{j\in \M}q_{\al(s) j}p_j \big(X(t)-G(X(s-\delta(s)), \al(s)) \big)   \\
\ad \ \ \qquad+ 2 p_{\al(s)}(Z(s))^\top f\big(X(s), X(s-\delta(s)), \al(s)\big) \\
\ad \ \ \qquad +p_{\al(s)} | g \big(X(s), X(s-\delta(s)), \al(s) \big)|^2\ \ \text{for} \ \  s\in [0, t_*).
\eea 
In view of \eqref{e.cond}, for any $s\in [0, t_*)$, 
\bea
{\cal H}(s)\ad \le (c-c_0) p_{\al(s)}  |Z(s)|^2 + c_1 p_{\al(s)}  |X(s)|^2+ c_2 p_{\al(s)}  |X(s-\delta(s))|^2.
\eea
Hence,
\beq{ee.11e}\notag
\barray
 \E\big[ {\cal H}(s) \big] \ad \le  (c-c_0) \E \big[ p_{\al(s)} |Z(s)|^2 \big] \\
\ad \qquad+ c_1\E\big[  p_{\al(s)} |X(s)|^2 \big]  + c_2\E\big[ p_{\al(s)}   |X(s-\delta(s))|^2\big] \ \ \text{for} \ \  s\in [0, t_*).
 \earray
\eeq
By virtue of Lemma \ref{lem:1}, for any $s\in [0, t_*)$,
 \beq{ee.12x} \notag
 \E \big[ p_{\al(s)} |X(s)|^2 \big]\le r_p\beta K_1  e^{-\lambda s}, \ \ \E \big[ p_{\al(s)} |X(s-\delta(s))|^2 \big] \le r_p\beta K_1  e^{-\lambda (s-\delta(s))}.
\eeq
Hence,
$$
\E\big[ {\cal H} (s) \big]  \le (c-c_0)K_1e^{-\lambda s} + c_1r_p\beta K_1e^{-\lambda s} + c_2 r_p \beta K_1e^{-\lambda s} \ \ \text{for} \ \  s\in [0, t_*).
$$
Combining this with the property of $\lambda$ in \eqref{ee.3b} yields \beq{ee.14e}
\barray
\disp\int_{0}^{t_*} e^{c s} \E\big[ {\cal H} (s) \big] ds  \ad  \le   K_1  \int_{0}^{t_*}e^{(c -\lambda)s} \Big[c - c_0 + c_1r_p\beta +c_2r_p\beta e^{\lambda r}\Big]ds  \\
\ad \le K_1  \int_{0}^{t_*}e^{(c -\lambda)s}  (c-\lambda)ds\\
\ad  = K_1  \big( e^{(c-\lambda)t_*} -1\big).
\earray
\eeq
It follows from   \eqref{ee.10e}  and \eqref{ee.14e} that
\beq{}\notag
\barray
e^{c t_*} \E \big[ p_{\al(t_*)} |Z(t_*)|^2 \big]  \ad \le \E\big[p_{i_0} |Z(0)|^2\big] + K_1  \big( e^{(c-\lambda)t_*} -1\big)\\
\ad <   K_1   e^{(c-\lambda)t_*},
\earray
\eeq
where we have used that $\Phi(0)=\E\big[p_{i_0} |Z(0)|^2\big]<K_1$; see \eqref{Phi0}. Consequently, $$\Phi(t_*) =  \E \big[ p_{\al(t_*)} |Z(t_*)|^2 \big]   < K_1 e^{-\lambda  t_*}=\Psi(t_*).$$
This contradicts the second statement in \eqref{ee.9d}. Consequently, \eqref{ee.8d} is verified.
By Lemma \ref{lem:1}, 
 Eq. \eqref{e.main} is mean-square exponentially stable and the conclusion follows.
  \end{proof}

\section{Stability of EM approximations}\label{sec:num} 

We proceed to construct the EM numerical solutions for Eq. \eqref{e.main} subject to the initial condition \eqref{e.main2}. Let the step size $\e$ be a fraction of $r$; specifically, we set $\e =r/N$ for some positive integer $N>r$. To approximate the switching component, let $\{\al_n\}_{n\ge 0}$ denote the discrete-time Markov chain characterized by the one-step transition probability matrix $e^{\e Q}$ with $\al_0=i_0$. For the detailed method of simulating the chain $\{\al_n\}$, we refer the reader to \cite[p. 112]{Mao2006}.

The application of the EM approximation method to Eq. \eqref{e.main} yields an approximation sequence $\{Y_n\}_{n\ge -N}$. To ensure that the EM approximation sequence is well-defined, we impose the following:
$$Y_{-(N+1)}=\xi(-N\e), \ \ \delta(-\e) = \delta(0), \ \ \al_{-1}=\al_0.$$
 Then
\beq{EM}
\barray
\ad Y_n = \xi (n \e), \ \  -N\le n\le 0,   \\ 
 \ad Y_{n+1} = Y_n +  G( Y_{n - R_{n}}, \al_{n}) - G( Y_{n-1 - R_{n-1}}, \al_{n-1} ) 
+  f \big( Y_n, Y_{n- R_n}, \al_n   \big)\e  
\\ 
\ad \ \ + g \big( Y_n, Y_{ n-R_n} , \al_n  \big) \Delta w_n\ \ \text{for} \ \  n\in \N_0,
 \earray
\eeq
where
\beq{e.wn}  \notag \Delta w_n = w((n+1)\e ) - w(n\e  ),\ \  R_n = \lfloor \delta (n\e)/\e\rfloor. \eeq

We recall that $\lfloor \delta (n\e)/\e\rfloor$ denotes the integer part of $\delta (n\e)/\e$. To explicitly indicate the dependence on the initial conditions $Y_n = \xi(n\e)$ for $-N \le n\le 0$ and $\al_0=i_0$, we denote the sequence $\{Y_n\}$ generated by \eqref{EM} as $\{Y^{\xi, i_0}_n\}$.

To facilitate the stability analysis of the numerical scheme, we first establish the following auxiliary results.

\begin{lem}  \label{lem:XZ}
Let $\lambda\in (0, -\frac{2}{r}\ln \kappa)$ and $\e\in (0, 1)$ be chosen such that $\lambda <-\frac{2\ln \ka}{r+\e}$. Define
$$\beta(\e)=\dfrac{1}{(1-\ka  e^{\lambda (r+\e)/2} )^2}.$$
Let $\{Y_n\}$ denote the EM approximation sequence, and let $Z_n=Y_n - G(Y_{n-1-R_{n-1}}, \al_{n-1})$ and $\lbar p_n = p_{\al_n}$ for $n\in \mathbb{N}_0$. Suppose further that there exist constants $k\in \mathbb{N}$ and $K>1$ such that
\beq{4.4}
\mathbb{E}\big[\lbar p_n |Z_n|^2 \big] \le Ke^{-\lambda n\e}\E\|\xi\|^2 \quad \text{for} \quad n=0, 1, \dots, k.
\eeq
Then the following assertions hold:
\begin{itemize}
\item[\rm (a)] For any $n= 0, 1, \dots, k$, 
\beq{4.4a}\mathbb{E} |Z_n|^2 \le \dfrac{1}{p_{\min}} Ke^{-\lambda n\e}\mathbb{E}\|\xi\|^2.\eeq
\item[\rm(b)]  For any $n=-N, \dots, 0, 1, \dots, k$,
\beq{4.4b}\mathbb{E}|Y_n|^2 \le \dfrac{1}{p_{\min}}\beta(\e) K e^{-\lambda n\e} \E \|\xi\|^2.
\eeq
\item[\rm (c)] For any $n= 0, 1, \dots, k$,
\beq{4.4c}
\mathbb{E}\big[\lbar p_n |Y_n|^2 \big] \le r_p\beta(\e) K e^{-\lambda n\e} \E \|\xi\|^2
\eeq
and
\beq{4.4d} \E \big[ \lbar p_n |Y_{n-1-R_{n-1}}|^2\big]\le r_p e^{\lambda (r+\e)}\beta(\e) Ke^{-\lambda n\e} \E \|\xi\|^2.\eeq
\end{itemize}
 \end{lem}
 
 \begin{proof}
(a) The inequality \eqref{4.4a} follows directly from the assumption \eqref{4.4} and the definition of $p_{\min}$.

(b) Regarding (b), subject to the validity of \eqref{4.4a}, the detailed proof of \eqref{4.4b} is provided in \cite[Lemma 3.1]{Ky3}.

(c) We proceed to prove (c). By virtue of \eqref{4.4b}, for any $n=0, 1, \dots, k$, we have
\beq{}\notag
\barray
\E\big[ \lbar p_n |Y_n|^2 \big]\ad \le p_{\max} \E |Y_n|^2 \\
\ad  \le \dfrac{   p_{\max} }{ p_{\min}} \beta(\e) K e^{-\lambda n\e} \mathbb{E}\|\xi\|^2 ;
\earray
\eeq
consequently, \eqref{4.4c} is verified.
 
(d) Finally, for assertion (d), in view of \eqref{4.4b}, it follows that
 \beq{}\notag
 \barray
  \E \big[ \lbar p_n |Y_{n-1-R_{n-1}}|^2\big] \ad \le p_{\max} \E  |Y_{n-1-R_{n-1}}|^2\\
  \ad \le  \dfrac{p_{\max}}{p_{\min}}\beta(\e) K e^{-\lambda (n-1-R_{n-1})\e} \E \|\xi\|^2\\
  \ad \le r_p e^{\lambda (r+\e)}\beta(\e) Ke^{-\lambda n\e} \E \|\xi\|^2.
 \earray
 \eeq
This yields the desired inequality \eqref{4.4d}. Note that in the last step, we have utilized the fact that $R_{n-1}\e\le r$. The proof is thus complete.
\end{proof}

\begin{rem}\label{rem:rem}
{\rm  For any $i, j\in \M$ and $\e\in (0, 1)$,
let $p_{ij}(\e)$ be the $(i, j)$ entry in the matrix $\exp(\e Q)$.
Define
 $$ \|{ Q}\|_{\infty}=\max\{|q_{ij}|: i, j=1, 2, \dots, m_0\}.$$
Note that
$$\exp(\e Q)={\bf I} +\e Q+\e^2 \sum\limits_{k=2}^\infty \dfrac{\e^{k-2} Q^k}{k!},$$
where ${\bf I}$ is the $m_0\times m_0$ identity matrix.
Let $\tau_{ij}(\e)$ be the $(i, j)$ entry in $\sum\limits_{k=2}^\infty \dfrac{\e^{k-2} Q^k}{k!}$.
It has been proved in \cite[Lemma 3.2]{Ky1} that
 \beq{n0}|\tau_{ij}(\e)|\le \exp({m_0 \|Q\|_\infty}) - 1- m_0 \|Q\|_\infty \ \ \text{for any} \ \ i, j \in \M, \ \ \e\in (0, 1).
\eeq
}
\end{rem}

We turn our attention to the mean-square exponential stability of the EM approximations. The following theorem demonstrates that, provided the step size $\e$ is sufficiently small, the EM approximations faithfully reproduce the mean-square exponential stability of the exact solution. Furthermore, it reveals that the numerical scheme is capable of reproducing the exponential decay rate $\lambda_0$ with arbitrary accuracy.

\begin{thm} \label{thm:2}
Suppose that the assumptions of Theorem \ref{thm:1} are satisfied. In addition, assume that there exists a constant $C_0>0$ such that
\beq{e.x}
|f(x, y, i)|^2\le C_0(|x|^2+|y|^2) \ \ \text{for any} \ \ x, y\in \R^d, \ \  i\in \M.
\eeq
Let $\lambda\in (0, \lambda_0)$ and define $\beta = (1-\ka e^{\lambda r/2})^{-2}$. Then there exists a constant $\e_0>0$ such that for any step size $\e\in (0, \e_0]$, the EM approximation sequence $\{Y_n\}$
 is mean-square exponentially stable. Moreover, we have
\beq{d.2}
\limsup\limits_{n\to \infty} \dfrac{1}{n\e } \ln \big( \E  |Y^{\xi, i_0}_n|^2  \big)\le -\lambda \ \ \text{for} \ \ (\xi, i_0)\in C^{b}_{\mathcal{F}_0}([-r, 0], \R^d)\times \M.
\eeq
\end{thm}

\begin{proof} 

The proof is divided into four steps. 

\noindent {\bf Step 1.}
 By virtue of \eqref{ee.3}, we can select a positive constant
 $ \e_0 < ({1}/{\lambda_0}) \vee  ( -\frac{2}{\lambda} \ln \ka -r)$ such that 
 \beq{eps}- c_0    
+ c_1 r_p\beta (\e_0)+ c_2  r_p e^{\lambda (r+\e_0)} +\e_0 C_1(\e_0)<-\lambda,
\eeq
 where
 $$\beta(\e) = \dfrac{1}{(1-\ka e^{\al (r+\e)/2})^2}$$
 and
$$C_1(\e) = \e r_0 \beta(\e)p_{\min}^{-1} \Big((2\|Q\|_\infty +c_1+C_0)  +  C_0e^{\lambda r} +  (2\|Q\|_\infty +c_2)e^{\lambda (1+r)\e}\Big), \ \ \e>0.$$
Consequently, it follows that
 \beq{d.4.2}
- c_0    
+ c_1 r_p\beta (\e)+ c_2  r_p e^{\lambda (r+\e)} +\e C_1(\e)<-\lambda \ \ \text{for} \ \ \e\in (0, \e_0].
\eeq
For the remainder of the proof, we fix $\e\in (0, \e_0]$. 
  Let $(\xi, i_0) \in C^b_{\mathcal{F}_0}([-r,0], \mathbb{R}^d) \times \mathcal{M}$.
For the sake of notational simplicity,  let $\{Y_n\} = \{Y^{\xi,i_0}_n\}$ be the approximation sequence given by \eqref{EM}, and $Z_n =Y_n - G(Y_{n-1-R_{n-1}}, \al_{n-1})$ for $n\in \mathbb{N}_0$. Without loss of generality, assume $\mathbb{E}\|\xi\|^2 > 0$.

\medskip

\noindent {\bf  Step 2:} 
Let us consider the real-valued sequences $\{\Phi_n\}_{n\ge0}$ and $\{\Psi_n\}_{n\ge0}$ defined by
\begin{equation*}
\Phi_n := \mathbb{E}\Big[\lbar{p}_n|Z_n|^2\Big], \quad \Psi_n := Ke^{-\lambda n\varepsilon}\mathbb{E}\|\xi\|^2 \quad \text{for} \quad n \in \mathbb{N}_0,
\end{equation*}
where $K_1 = K\mathbb{E}\|\xi\|^2$ and $K = 2(1+\kappa^2)\sum_{i\in \M}p_i$. Then $\Phi_0 \le \Psi_0$.
We aim to show that
\beq{4.16}
\Phi_n = \mathbb{E}\Big[\lbar{p}_n|Z_n|^2\Big] \le Ke^{-\lambda n\varepsilon}\mathbb{E}\|\xi\|^2 = \Psi_n \ \  \text{for} \ \  n \in \mathbb{N}_0.
\eeq
Were this statement false, there would exist $k \in \mathbb{N}_0$ such that
\beq{4.17}
\Phi_n \le \Psi_n \text{ for } 0 \le n \le k, \quad \Phi_{k+1} > \Psi_{k+1}.
\eeq
Recalling \eqref{EM}, we have
$$
Z_{k+1} = Z_k + f(Y_k, Y_{k-R_k}, \alpha_k)\varepsilon + g(Y_k, Y_{k-R_k}, \alpha_k)\Delta w_k. $$
Taking the squared norm on both sides and multiplying by $\lbar{p}_{k+1}$, we obtain
\beq{4.20}
\barray
 \lbar{p}_{k+1}|Z_{k+1}|^2 \ad = \lbar{p}_{k+1}|Z_k|^2 + \lbar{p}_{k+1}|f(Y_k, Y_{k-R_k}, \alpha_k)\varepsilon|^2 + \lbar{p}_{k+1}|g(Y_k, Y_{k-R_k}, \alpha_k)\Delta w_k|^2 \\
\ad \quad + 2\lbar{p}_{k+1}Z_k^\top f(Y_k, Y_{k-R_k}, \alpha_k)\varepsilon + 2\lbar{p}_{k+1}Z_k^\top g(Y_k, Y_{k-R_k}, \alpha_k)\Delta w_k \\
\ad \quad + 2\lbar{p}_{k+1}\big(f(Y_k, Y_{k-R_k}, \alpha_k)\varepsilon\big)^\top g(Y_k, Y_{k-R_k}, \alpha_k)\Delta w_k.
\earray
\eeq
It is readily verified that
\beq{4.21}
\barray
\ad \mathbb{E}\Big[\lbar p_{k+1} |g(Y_k, Y_{k-R_k}, \alpha_k)\Delta w_k|^2 \Big]= \varepsilon\mathbb{E}\Big[ p_{k+1}|g(Y_k, Y_{k-R_k}, \alpha_k)|^2\Big],\\
\ad \mathbb{E}\Big[ 2 \lbar p_{k+1} Z_k^\top g(Y_k, Y_{k-R_k}, \alpha_k)\Delta w_k\Big]= 0,\\
\ad \mathbb{E}\Big[ 2 f(Y_k, Y_{k-R_k}, \alpha_k)^T g(Y_k, Y_{k-R_k}, \alpha_k)\Delta w_k\Big] = 0. 
\earray
\eeq
Taking expectations on both sides of \eqref{4.20} and substituting \eqref{4.21} lead to 
\beq{4.22}
\barray
\E\Big[\lbar{p}_{k+1}|Z_{k+1}|^2\Big] \ad = \E\Big[\lbar{p}_{k+1}|Z_k|^2\Big] + \varepsilon^2\E\Big[\lbar{p}_{k+1}|f(Y_k, Y_{k-R_k}, \alpha_k)|^2\Big] \\
&\quad + \e\E\Big[\lbar{p}_{k+1}\Big(2 Z_k^\top f(Y_k, Y_{k-R_k}, \alpha_k) + |g(Y_k, Y_{k-R_k}, \alpha_k)|^2\Big) \Big].
\earray
\eeq

\medskip

\noindent {\bf Step 3:}
We proceed to estimate the terms on the right hand side of \eqref{4.22}.
By the property of conditional expectation and \eqref{e.main3}, we have
\beq{ee.12}
\barray
\E \Big[\lbar p_{k+1} | Z_k|^2  {\bf1}_{\{\al_{k+1}=\al_k\}} \Big]\ad = \E\bigg[\E \Big[\lbar p_{k} | Z_k|^2  {\bf1}_{\{\al_{k+1}=\al_k\}}  \Big| \al_k\Big] \bigg]  \\
\ad = \E\bigg[\E \Big[\lbar p_{k} | Z_k|^2 \Big| \al_k \Big]  \E \Big[ {\bf1}_{\{\al_{k+1}=\al_k\}} | \al_k \Big]  \bigg]\\
\ad = \E\bigg[\E \Big[\lbar p_{k} | Z_k|^2 \Big| \al_k \Big]
\sum\limits_{i\in \M}  {\bf1}_{\{\al_{k}=i\}} \big( 1 + q_{ii}\e + \e^2\tau_{ii}(\e)
\big)
  \bigg].
\earray
\eeq
In the derivation above, we have utilized the fact that
${\bf1}_{\{\al_{k+1}=\al_k\}}$ and $\mathcal{F}_{{k}\e}$ are independent given $\al_k$ and the fact that $\P (\al_{k+1}=i|\al_k=i)= 1 +q_{ii}\e + \e^2\tau_{ii}(\e)$ for any $i\in \M$. Consequently, \eqref{ee.12} impliest
\beq{ee.13}
\barray
\E \Big[\lbar p_{k+1} | Z_k|^2  {\bf1}_{\{\al_{k+1}=\al_k\}} \Big] \ad =\sum\limits_{i\in \M} \E\Big[ \lbar p_{k} | Z_k |^2
  {\bf1}_{\{\al_{k}=i\}}  \Big]+\sum\limits_{i\in \M} \E\Big[ p_i |  Z_k |^2
 {\bf1}_{\{\al_{n}=i\}} \big( q_{ii}\e + \e^2\tau_{ii}(\e)
\big)
  \Big]\\
  \ad  \qquad =\E\Big[ \lbar p_{k} | Z_k |^2    \Big] +\sum\limits_{i\in \M} \E\Big[p_i | Z_k |^2
 {\bf1}_{\{\al_{k}=i\}} \big( q_{ii}\e + \e^2\tau_{ii}(\e)
\big)
  \Big].
 \earray
\eeq
Similarly,
\beq{ee.14}
\barray
\E \Big[\lbar p_{k+1} | Z_k|^2  {\bf1}_{\{\al_{k+1}\ne\al_k\}} \Big]=\sum\limits_{i\in \M}\sum\limits_{j\ne i} \E\Big[p_j | Z_k |^2
 {\bf1}_{\{\al_{k}=i\}} \big( q_{ij}\e + \e^2\tau_{ij}(\e)
\big)
  \Big].
\earray
\eeq
In view of \eqref{ee.13} and \eqref{ee.14},
\beq{}\notag
\barray
 \E \Big[\lbar p_{k+1} | Z_k |^2 \Big]   \ad =\E\Big[ \lbar p_{k} | Z_k |^2    \Big] +\sum\limits_{i\in \M}\sum\limits_{j\in \M} \E\Big[r_j | Z_k |^2
 {\bf1}_{\{\al_k=i\}} \big( q_{ij}\e + \e^2\tau_{ij}(\e)
\big)
  \Big]\\
  \ad   =\E\Big[ \lbar p_{k} | Z_k |^2    \Big] + \sum\limits_{j\in \M} \E\Big[p_j | Z_k |^2
 \big( q_{\al_k j}\e + \e^2\tau_{\al_k j}(\e)
\big)
  \Big];
 \earray
\eeq
that is,
\beq{hh.1}
\E\Big[ \lbar p_{k+1} |Z_k|^2\Big]=\E\Big[ \lbar p_{k} |Z_k|^2\Big] + 
\sum\limits_{j\in \M} \E\Big[p_j | Z_k |^2
 \big( q_{\al_k j}\e + \e^2\tau_{\al_k j}(\e)
\big)
  \Big].
\eeq
Similarly,
\beq{hh.2}
\barray
\ad \E\Big[\lbar{p}_{k+1}\Big( Z_k^\top f(Y_k, Y_{k-R_k}, \alpha_k) + |g(Y_k, Y_{k-R_k}, \alpha_k)|^2\Big) \Big]\\ \ad \quad =\E\Big[\lbar{p}_{k}\Big(2 Z_k^\top f(Y_k, Y_{k-R_k}, \alpha_k) + |g(Y_k, Y_{k-R_k}, \alpha_k)|^2\Big) \Big]\\
\ad \qquad + \sum\limits_{j\in \M} \E\Big[\lbar p_{j}
\Big( 2Z_k^\top f(Y_k, Y_{k-R_k}, \alpha_k) + |g(Y_k, Y_{k-R_k}, \alpha_k)|^2\Big) 
    \big( q_{\al_k j}\e + \e^2\tau_{\al_k j}(\e)
\big) \Big].
\earray
\eeq
Substituting \eqref{hh.1} and \eqref{hh.2} into \eqref{4.22}, we arrive at
 \beq{hh.3}
\barray
\ad \E\Big[\lbar{p}_{k+1}|Z_{k+1}|^2\Big]  = \E\Big[\lbar{p}_{k}|Z_k|^2\Big]   
\\
\ad \quad + \e\E\Big[\lbar{p}_{k}\Big(2 Z_k^\top f(Y_k, Y_{k-R_k}, \alpha_k) + |g(Y_k, Y_{k-R_k}, \alpha_k)|^2\Big)+ \sum\limits_{j\in \M} \big( p_j  |Z_k|^2 
  q_{\al_k j}\big)\Big]\\
  \ad \quad + \varepsilon^2\E\Big[\lbar{p}_{k+1}|f(Y_k, Y_{k-R_k}, \alpha_k)|^2\Big]
  \\
\ad \quad + \e\sum\limits_{j\in \M} \E\Big[p_j \Big(  2Z_k^\top f(Y_k, Y_{k-R_k}, \alpha_k) + |g(Y_k, Y_{k-R_k}, \alpha_k)|^2\Big) 
 \big( q_{\al_k j}\e + \e^2\tau_{\al_k j}(\e)
\big)
  \Big].
\earray
\eeq

\bigskip

\noindent {\bf  Step 4:}
We proceed to estimate the terms on the right hand side of \eqref{hh.3}.
Let 
$$r_0 = \big(\sum_{j\in \M} p_j \big)\Big(\exp({m_0 \|Q\|_\infty}) - (m_0-1) \|Q\|_\infty +1 \Big);$$
see Remark \ref{rem:rem}.
Noting that $\e\in (0,1)$, we have $\e^2<\e$.
Furthermore, in view of \eqref{n0}, we have \beq{n1}\notag
\barray
|q_{\al_{k}j}\e  + \e^2\tau_{\al_{k}j}(\e)| \ad \le (\|Q\|_\infty + |\tau_{\al_{k}j}(\e)|)\e\\
\ad \le \Big(\exp({m_0 \|Q\|_\infty}) - 1- (m_0-1) \|Q\|_\infty\Big)\e.
\earray
\eeq
Hence,
$$\sum_{j\in \M} p_j  |q_{\al_{k}j}\e  + \e^2\tau_{\al_{k}j}(\e)| \le r_0 \e.$$
Using \eqref{e.x} together with the definition of $r_0$, we obtain
\beq{hh.4}
\e^2\E\Big[\lbar{p}_{k+1}|f(Y_k, Y_{k-R_k}, \alpha_k)|^2\Big]\le \e^2 r_0 C_0 \E \Big[ |Y_k|^2 + |Y_{k-R_k}|^2 \Big]. 
\eeq
We observe that
 $$|Z_k|^2\le (1+\ka^2)(|Y_k|^2 + |Y_{k-1-R_{k-1}}|^2)\le 2|Y_k|^2 + 2|Y_{k-1-R_{k-1}}|^2.$$
 It follows from \eqref{e.cond} that
 \beq{}\notag
 \barray
 \ad 2Z_k^\top f(Y_k, Y_{k-R_k}, \alpha_k)+ |g(Y_k, Y_{k-R_k}, \alpha_k)|^2 \\
 \ad \quad \le \|Q\|_\infty |Z_k|^2 + c_1 |Y_k|^2 + c_2|Y_{k-1-R_{k-1}}|^2\\
 \ad \quad \le (2\|Q\|_\infty +c_1)|Y_k|^2 + (2\|Q\|_\infty +c_2)|Y_{k-1-R_{k-1}}|^2.
 \earray
 \eeq
Hence
\beq{hh.5}
\barray
 \ad\e\sum\limits_{j\in \M} \E\Big[p_j \Big( 2Z_k^\top f(Y_k, Y_{k-R_k}, \alpha_k) + |g(Y_k, Y_{k-R_k}, \alpha_k)|^2\Big) 
 \big( q_{\al_k j}\e + \e^2\tau_{\al_k j}(\e)
\big)
  \Big]\\
  \ad \ \ \le \e^2  r_0 \Big( (2\|Q\|_\infty +c_1)\E|Y_k|^2 + (2\|Q\|_\infty +c_2)\E|Y_{k-1-R_{k-1}}|^2\Big).  
\earray
\eeq
Utilizing \eqref{hh.4} and \eqref{hh.5} in \eqref{hh.3}, we deduce
\beq{jk-2}
\barray
\ad \E\Big[\lbar{p}_{k+1}|Z_{k+1}|^2\Big]  \le \E\Big[\lbar{p}_{k}|Z_k|^2\Big]   
\\
\ad \quad +\e\E\Big[\lbar{p}_{k}\Big(2 Z_k^\top f(Y_k, Y_{k-R_k}, \alpha_k) + |g(Y_k, Y_{k-R_k}, \alpha_k)|^2\Big)+ \sum\limits_{j\in \M} \big( p_j  |Z_k|^2 
  q_{\al_k j}\big)\Big]\\
  \ad \quad +\e^2  r_0 \Big( (2\|Q\|_\infty +c_1+C_0)\E|Y_k|^2  + C_0\E|Y_{k-R_k}|^2+ (2\|Q\|_\infty +c_2)\E|Y_{k-1-R_{k-1}}|^2\Big). \earray
\eeq
By \eqref{e.cond}, 
\beq{}\notag
\barray
\ad \lbar{p}_{k}\Big(2 Z_k^\top f(Y_k, Y_{k-R_k}, \alpha_k) + |g(Y_k, Y_{k-R_k}, \alpha_k)|^2\Big)  + \sum\limits_{j\in \M} \big( p_j | Z_k |^2
  q_{\al_k j}\big)\\
  \ad \quad \le -c_0 \lbar p_k |Z_k|^2 + c_1 \lbar p_k|Y_k|^2 + c_2  \lbar p_k|Y_{k-1-R_{k-1}}|^2.
\earray
\eeq
Hence, 
\beq{jk-3}
\barray
\ad \e\E\Big[\lbar{p}_{k}\Big(2 Z_k^\top f(Y_k, Y_{k-R_k}, \alpha_k) + |g(Y_k, Y_{k-R_k}, \alpha_k)|^2\Big)+ \sum\limits_{j\in \M} \big( p_j  |Z_k|^2 
  q_{\al_k j}\big)\Big]\\
  \ad \qquad \le \e\bigg(-c_0 \E\Big[ \lbar p_k |Z_k|^2\Big]  + c_1 \E\Big[ \lbar p_k |Y_k|^2\Big] + c_2 \E \Big[ \lbar p_k |Y_{k-1-R_{k-1}}|^2\Big]\bigg).
   \earray
\eeq
Substituting \eqref{jk-3} into \eqref{jk-2} yields
\beq{}\notag
\barray
 \ad\E\Big[\lbar{p}_{k+1}|Z_{k+1}|^2\Big]   \le  (1-\e c_0)\E\Big[\lbar{p}_{k}|Z_k|^2\Big]   
+ \e c_1 \E\Big[ \lbar p_k |Y_k|^2\Big] + \e c_2 \E \Big[ \lbar p_k |Y_{k-1-R_{k-1}}|^2\Big]\\
  \ad \quad +\e^2  r_0 \Big( (2\|Q\|_\infty +c_1+C_0)\E|Y_k|^2  + C_0\E|Y_{k-R_k}|^2+ (2\|Q\|_\infty +c_2)\E|Y_{k-1-R_{k-1}}|^2\Big). \earray
\eeq
Applying Lemma \ref{lem:XZ}, we find
\beq{} \notag
\barray
 \ad \E\Big[\lbar{p}_{k+1}|Z_{k+1}|^2\Big]   \le  (1-\e c_0) K_1 e^{-\lambda k\e}    
+ \e c_1 K_1 r_p\beta (\e) e^{-\lambda k\e}  + \e c_2 K_1 r_p \beta(\e) e^{\lambda (r+\e)}e^{-\lambda k\e}\\
  \ad \quad + \e^2 r_0 \beta(\e)K_1 p_{\min}^{-1} e^{-\lambda k\e}\Big((2\|Q\|_\infty +c_1+C_0)  +  C_0e^{\lambda R_{k}\e} +  (2\|Q\|_\infty +c_2)e^{\lambda(1+R_{k-1})\e}\Big)\\
  \ad\le  K_1 e^{-\lambda k\e} \bigg( 1+ \e \Big[ - c_0    
+  c_1 r_p\beta (\e)+ c_2  r_p e^{\lambda (r+\e)} +\e C_1(\e)\Big]\bigg),
  \earray
\eeq
where $$C_1(\e) = \e r_0 \beta(\e)p_{\min}^{-1} \Big((2\|Q\|_\infty +c_1+C_0)  +  C_0e^{\lambda r} +  (2\|Q\|_\infty +c_2)e^{\lambda (1+r)\e}\Big).$$
By \eqref{d.4.2}, we have
\beq{jk-20}
\E\Big[\lbar{p}_{k+1}|Z_{k+1}|^2\Big] \le K_1 e^{-\lambda k\e}(1-\e\lambda ).
\eeq
Since $\e<\e_0<1/\lambda_0$ and $\lambda <\lambda_0$, then $0<1-\e\lambda <e^{-\lambda \e}$. It follows from \eqref{jk-20} that 
$$\E\Big[\lbar{p}_{k+1}|Z_{k+1}|^2\Big] < K_1 e^{-\lambda (k+1)\e};$$
that is, $\Phi_{k+1}<\Psi_{k+1}$. 
This stands in contradiction to the second inequality in \eqref{4.17}. 
As a result, \eqref{4.16} holds; that is,
$$
\mathbb{E}\Big[\lbar{p}_n|Z_n|^2\Big] \le Ke^{-\lambda n\varepsilon}\mathbb{E}\|\xi\|^2 \ \  \text{for} \ \  n \in \mathbb{N}_0.$$
By Lemma \ref{lem:XZ}, 
$$
\mathbb{E}|Y_n|^2 \le p^{-1}_{\min} \beta(\e) Ke^{-\lambda n\varepsilon}\mathbb{E}\|\xi\|^2 \ \  \text{for} \ \  n \in \mathbb{N}_0.$$
Thus, we conclude that the EM approximation \eqref{EM} is mean-square exponentially stable, and the inequality \eqref{d.2} is satisfied.
\end{proof}

\begin{rem}\label{rem:eps}
{\rm 
Let the assumptions of Theorem \ref{thm:2} hold. Then
the proof of Theorem \ref{thm:2}  provides guidance for choosing the step size $\e$ in the numerical experiments.
Indeed, let a positive constant
$ \e_0 < ({1}/{\lambda_0}) \vee  ( -\frac{2}{\lambda} \ln \ka -r)$ that satisfies condition \eqref{eps}. Then for any $\e\in (0, \e_0]$, the EM approximation $\{Y_n\}$ given by \eqref{EM} is mean-square exponentially stable.
}
\end{rem}

To illustrate the theoretical results derived in the previous sections, we present the following numerical example.

\begin{exm}\label{exam:1}
{\rm Consider the scalar NSDDE given by
\beq{yy.1}
\barray
d \big( X(t) - G(X(t-\delta(t), \al(t)) \big)\ad  = f \big( X(t), X(t-\delta(t)),  \al(t)\big) dt \\
\ad \ \  + g\big( X(t), X(t-\delta(t)),  \al(t) \big)dw(t)  \ \ \text{for} \ \ t\ge 0,
\earray
\eeq
where the state space of the Markov chain $\al(t)$ is $\M=\{1, 2\}$ and the time-dependent delay is $\delta(t)=2-\cos t$ for $t\ge 0$. The coefficients of Eq. \eqref{yy.1} are specified as
\beq{}\notag
\barray
\aad G(y, 1) = 0.05y, \ \ f(x, y, 1) = -1.1x + 0.2y, \ \ g(x, y, 1) = 0.3x\cos y,\\
\aad G(y,2) = 0.1 y, \ \ f(x, y, 2)= 0.2 x +0.1y, \ \ g(x, y, 2) = 0.2\sin y 
\earray
\eeq
for $(x, y)\in \mathbb{R}^2$. Note that here 
the delay bound is
$r=3$ and the contraction constant is $\ka=0.1$. 

It is instructive to examine the subsystems associated with Eq. \eqref{yy.1}. They correspond to the two states of the Markov chain:
\beq{yy.2}
\barray
 \ad d \big( X^{(1)}(t) - 0.05 X^{(1)}(t-\delta(t)) \big)   =  \big( -1.1X^{(1)}(t) + 0.2X^{(1)}(t-\delta(t))\big) dt \\
\ad \qquad\qquad  + 0.3X^{(1)}(t) \cos ( X^{(1)}(t-\delta(t)) ) dw(t),
\earray
\eeq
and
\beq{yy.3}
\barray
 \ad d \big( X^{(2)}(t) - 0.1 X^{(2)}(t-\delta(t)) \big)  =  \big( 0.2X^{(2)}(t) + 0.1X^{(2)}(t-\delta(t))\big) dt \\
 \ad\qquad \qquad+ 0.2 \sin ( X^{(2)}(t-\delta(t)) ) dw(t).
\earray
\eeq
The system switches between these two modes according to the evolution of $\al(t)$.

\begin{figure}[htbp!]
     \centering
  \includegraphics[scale=0.6]{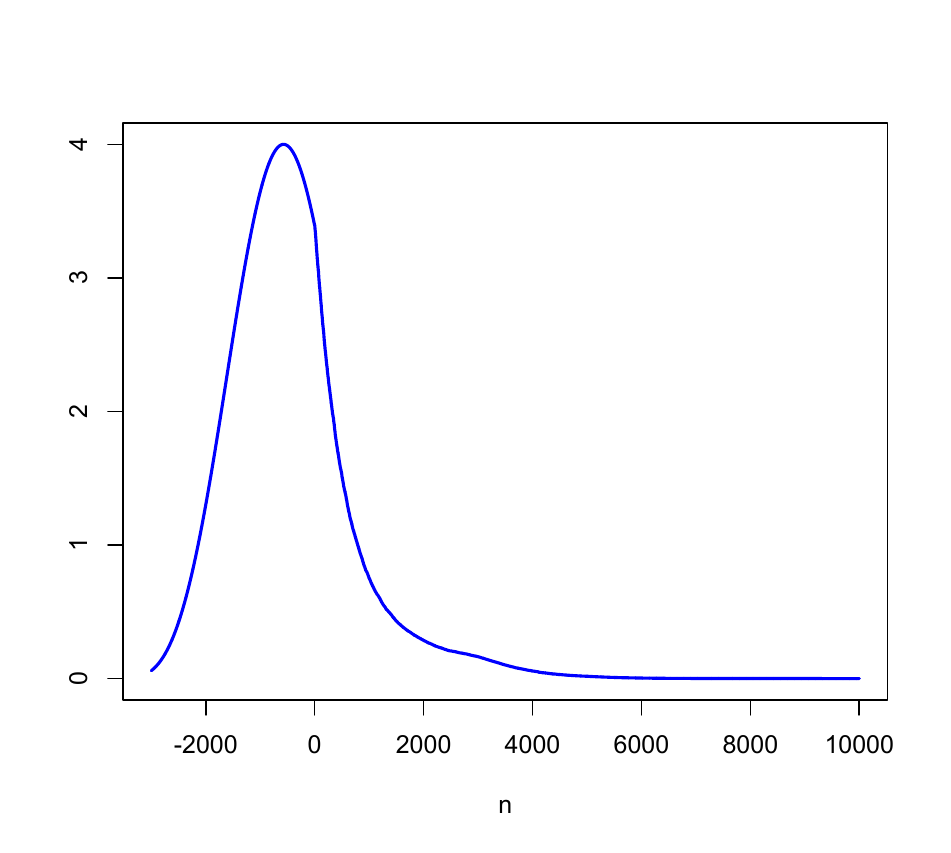}
    \caption{ Graph of $\{\E|Y^{(1)}_n|^2\}$ for the EM approximation $\{Y^{(1)}_n\}$ of Eq. \eqref{yy.2}. The expectation is estimated by averaging over 100 simulations. }
          \label{Fig1}
          \end{figure}

          \begin{figure}[htbp!]
     \centering
  \includegraphics[scale=0.6]{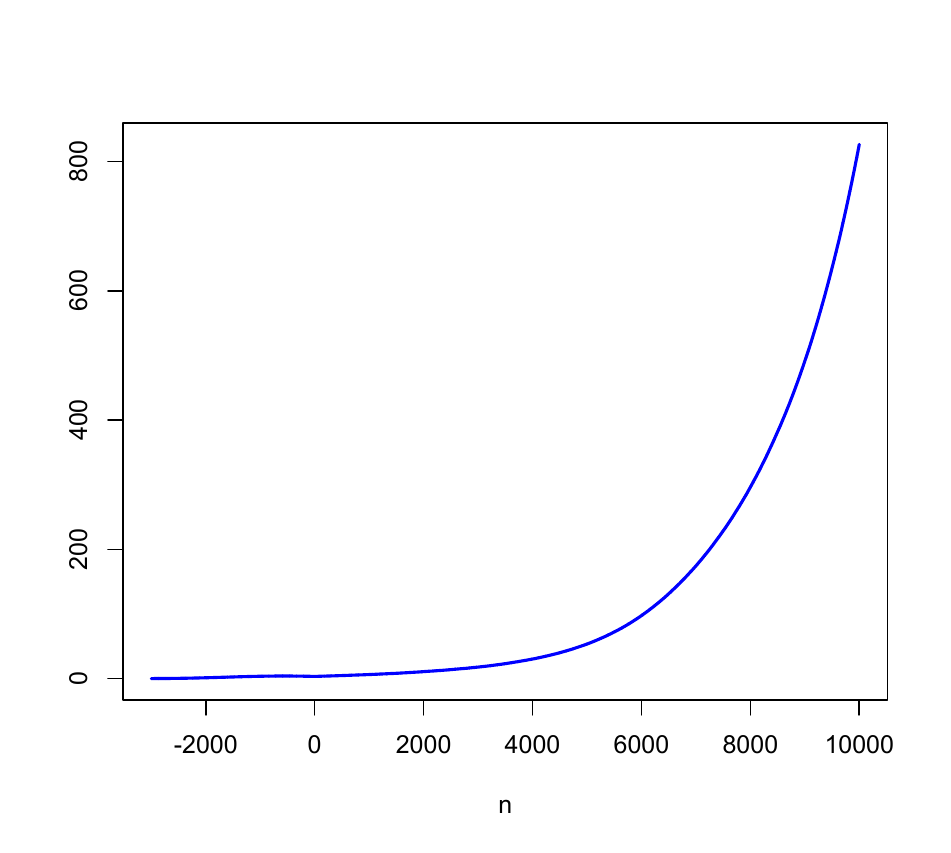}
    \caption{ Graph of $\{\E|Y^{(2)}_n|^2\}$ for the EM approximation $\{Y^{(2)}_n\}$ of Eq. \eqref{yy.3}. The expectation is estimated by averaging over 100 simulations. }
          \label{Fig2}
          \end{figure}

Note that Eq. \eqref{yy.2} and Eq. \eqref{yy.3} can be viewed as special cases of Eq. \eqref{e.main}. Let $\{Y^{(1)}_n\}$ and $\{Y^{(2)}_n\}$ denote the EM approximation sequences for $X^{(1)}(t)$ and $X^{(2)}(t)$, respectively. 
For illustration purposes, we display the graphs of the sequences $\{\E|Y^{(1)}_n|^2\}$ and $\{\E|Y^{(2)}_n|^2\}$ in Figure \ref{Fig1} and Figure \ref{Fig2}, using the initial condition $\xi(t)=1+\sin (1-t)$. The expectation $\E|Y^{(i)}_n|^2$ is approximated by averaging over 100 simulation paths. What do these numerical results suggest? They indicate that Eq. \eqref{yy.2} is likely mean-square exponentially stable, whereas Eq. \eqref{yy.3} appears to be unstable.

Suppose the generator of $\al(t)$ is given by {\footnotesize $Q = \begin{bmatrix} -1 & 1 \\ 3 & -3  \end{bmatrix}$}. 
Our task now is to examine the mean-square exponential stability of the exact solutions and numerical approximations of the coupled system \eqref{e.main}. It is worth noting that in this case, $\sup_{t\ge 0} \delta'(t)=1$. Consequently, the existing stability criteria in the literature—particularly those in \cite{Yuan2015,M2017S,Zhang2019}—are not applicable.

We proceed by verifying assumption (H). Let us choose $p_1=0.5$ and $p_2=1$. We compute:
\beq{yy.4}
\barray
\ad 2 \big(x-G(y, 1)\big) f(x, y, 1)  + |g(x, y, 1)|^2 + \sum\limits_{j\in \M} q_{1j} \dfrac{p_j}{p_1} |x-G(y, 1)|^2\\
\ad \ \ = 2(x-0.05y)(-1.1x+0.2y) + |0.3x\cos y|^2 -(x-0.05y)^2 +2(x-0.05y)^2 \\
\ad \ \ \le 2(x-0.05y)(-1.1x + 0.055y +0.145y) + 0.09x^2 + (x-0.05y)^2\\
\ad \ \ = -1.2(x-0.05y)^2 + 2\cdot 0.145xy - 0.0145 y^2 + 0.09x^2\\
\ad \ \ \le  -1.2(x-0.05y)^2 + 0.145x^2 + 0.145y^2 - 0.0145 y^2 + 0.09x^2\\
\ad \ \ = -1.2(x-0.05y)^2 + 0.235x^2 + 0.1305y^2
\earray
\eeq
and
\beq{yy.5}
\barray
\ad 2 \big(x-G(y, 2)\big) f(x, y, 2)  + |g(x, y, 2)|^2 + \sum\limits_{j\in \M} q_{2j} \dfrac{p_j}{p_2} |x-G(y, 2)|^2\\
\ad \ \ = 2(x-0.1y)(0.2x+0.1y) + |0.2\sin y|^2 +1.5(x-0.1y)^2 -3(x-0.1y)^2 \\
\ad \ \ \le 2(x-0.1y)(0.2x -0.02y  +0.12y) + 0.04y^2 -1.5 (x-0.1y)^2\\
\ad \ \ = 0.4(x-0.1y)^2 + 0.24 xy - 0.024 y^2 + 0.04y^2-1.5 (x-0.1y)^2\\
\ad \ \ \le  -1.1(x-0.1y)^2 + 0.12x^2 + 0.12y^2 - 0.024 y^2 + 0.04y^2\\
\ad \ \ = -1.1(x-0.1y)^2 + 0.12x^2 + 0.136y^2.
\earray
\eeq
In view of \eqref{yy.4} and \eqref{yy.5}, assumption (H) is satisfied with the parameters
$$p_1=0.5, \ \ p_2=1, \ \ c_0=1.1, \ \ c_1 = 0.235, \ \ c_2= 0.136.$$
Hence, we calculate $r_p=\max\{\frac{p_2}{p_1}, \frac{p_1}{p_2}\}=2$. Then
$$-c_0 +\dfrac{r_p}{(1-\ka)^2}(c_1+ c_2)=-1.1 + \dfrac{2 (0.235 + 0.136)}{(1-0.1)^2}<0;$$
that is, condition \eqref{ee.3} is satisfied.
By virtue of Theorem \ref{thm:1}, we conclude that Eq. \eqref{yy.1} is mean-square exponentially stable. Furthermore, Theorem \ref{thm:2} implies that for a sufficiently small step size $\e$, the EM approximations of Eq. \eqref{yy.1} preserve this stability.

        \begin{figure}[htbp!]
     \centering
  \includegraphics[scale=0.6]{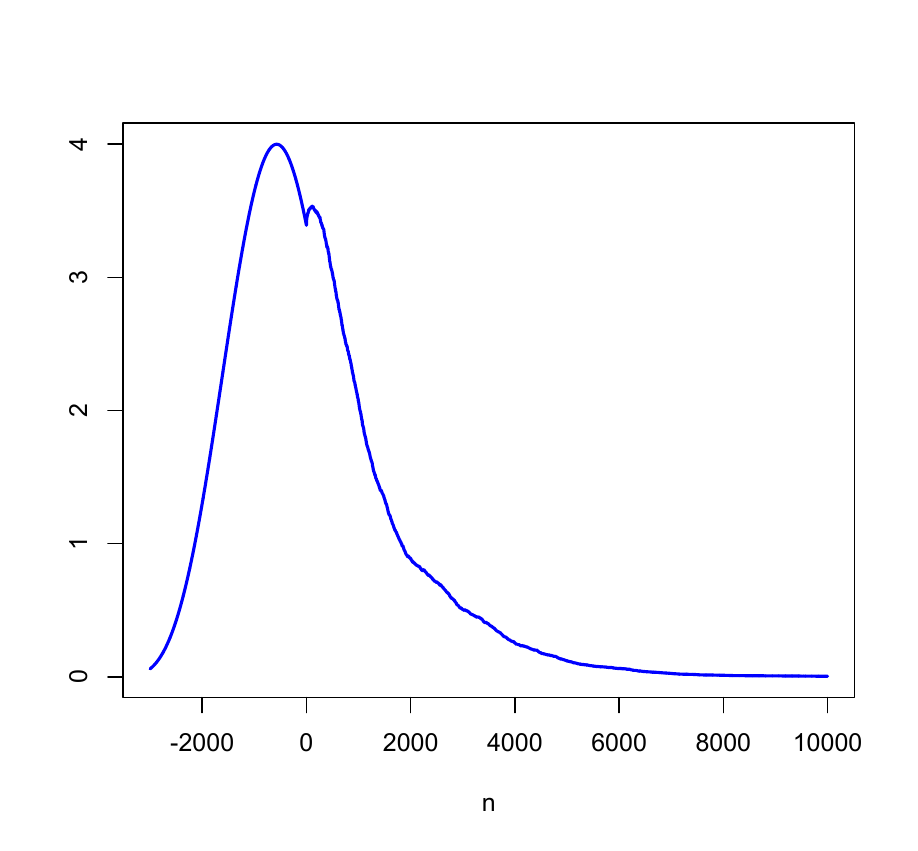}
   \caption{
  Graph of $\{\E|Y^{\xi, i_0}_n|^2\}$ for the EM approximation $\{Y^{\xi, i_0}_n\}$ of
  Eq. \eqref{yy.1}. The expectation is estimated by averaging over 100 simulations.
   }        \label{Fig3}
     \end{figure}

     \begin{figure}
         \centering    \includegraphics[scale=0.6]{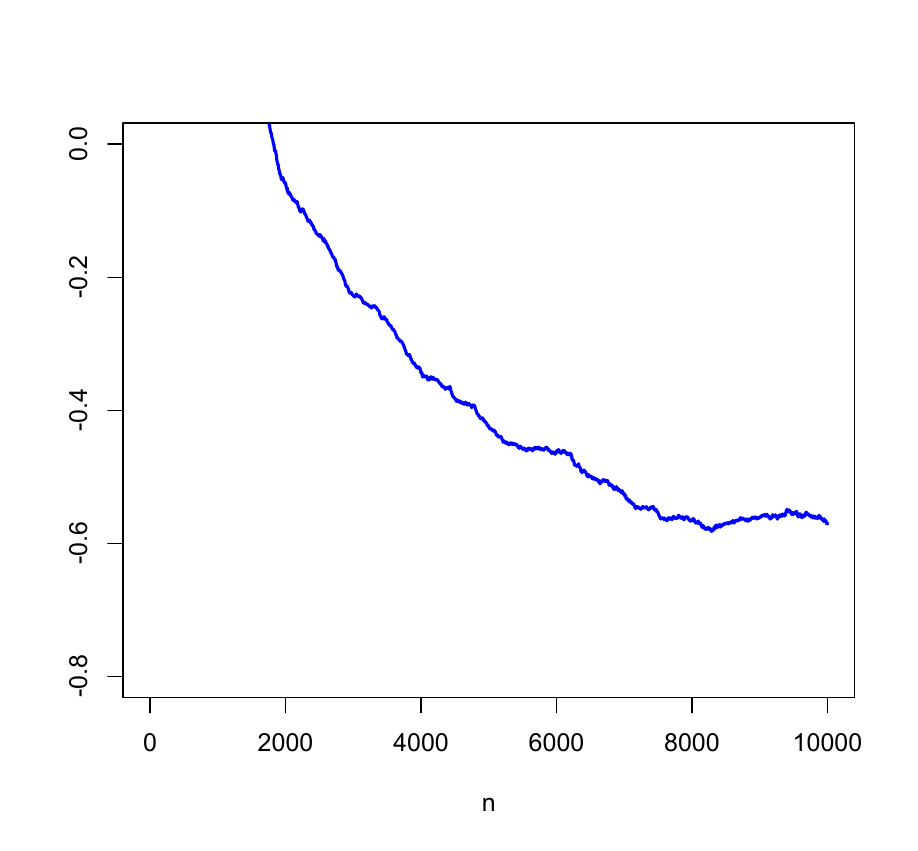}
         \caption{
         Graph of $\big\{ \frac{1}{n\e}\ln\big(\E|Y^{\xi, i_0}_n|^2\big)\big\}$ for the EM approximation $\{Y^{\xi, i_0}_n\}$ of Eq. \eqref{yy.1}. The expectation is estimated by averaging over 100 simulations.
         }        \label{Fig4}
\end{figure}

Utilizing the EM approximation method with step size $\e=0.001$, initial condition $\xi(t)=1+\sin (1-t)$ for $t\in [-3, 0]$, and $i_0=2$, we present the graph of the sequence $\{\E|Y^{\xi, i_0}_n|^2\}$ in Figure \ref{Fig3}. 
(The expectation is approximated by averaging over 100 simulations). We verify that the chosen step size $\e = 0.001$ satisfies the selection guidelines outlined in Remark \ref{rem:eps}.
To further illustrate the behavior, the corresponding graph for the sequence $ \big\{\frac{1}{n\e}\ln \big(\E|Y^{\xi, i_0}_n|^2\big)\big\}$ is shown in Figure \ref{Fig4}. These graphical representations are consistent with our theoretical analysis, demonstrating that $\E|Y^{\xi, i_0}_n|^2$ exponentially converges to zero as $n$ increases.

This example serves to highlight the stabilization effect of the Markovian switching component. Although the subsystem \eqref{yy.3} appears unstable in isolation, the introduction of Markovian switching results in a system \eqref{yy.1} that is mean-square exponentially stable.}

\end{exm}

\section{Concluding remarks} \label{sec:con}

In this paper, we have investigated the mean-square exponential stability of NSDDEs with Markovian switching.  
The essence of our contribution lies in the development of a novel comparison principle via contradiction for NSDDEs with Markovian switching.
By employing this approach, we have established stability criteria for both exact solutions and Euler-Maruyama approximations without imposing restrictive differentiability conditions on the time-varying delay. Consequently, our results cover a broader class of delay functions than those typically found in the existing literature, demonstrating that the numerical scheme can faithfully reproduce the stability and decay rate of the underlying system.

This work opens up several avenues for future research. For instance, our stability analysis could likely be extended to other numerical methods, such as the theta-approximations presented in \cite{Ky2}. Although this paper focused on NSDDEs with a single time-delay, the methodology developed herein appears applicable to more general settings. We believe it can be adapted to treat NSDDEs with multiple delays or even more general neutral stochastic functional differential equations.
 Furthermore, it would be interesting and challenging to apply our techniques to analyze the stability of truncated Euler–Maruyama approximations for NSDDEs and SDDEs, as developed in \cite{Guo2018}. Addressing these problems effectively may require the development of new analytic tools, and we plan to explore these research directions in our future work.


\begin{thebibliography}{99}
{\small
\setlength{\baselineskip}{0.10in}
\parskip=-1pt

\bibitem{Guo2018}
Q. Guo, X. Mao, R. Yue, The truncated Euler--Maruyama method for stochastic differential delay equations, \emph{Numer. Algorithms} 78 (2018) 599--624.

\bibitem{K2013}
V. Kolmanovskii, A. Myshkis, \emph{Introduction to the Theory and Applications of Functional Differential Equations}, vol. 463, Springer Science \& Business Media, 2013.

\bibitem{Yuan2015}
G. Lan, C. Yuan, Exponential stability of the exact solutions and $\theta$-EM approximations to neutral SDDEs with Markov switching, \emph{J. Comput. Appl. Math.} 285 (2015) 230--242.

\bibitem{Zhu16}
L. Liu, Q. Zhu, Mean square stability of two classes of theta method for neutral stochastic differential delay equations, \emph{J. Comput. Appl. Math.} 305 (2016) 55--67.

\bibitem{Mao2023}
W. Mao, B. Chen, F. Wu, X. Mao, On the existence and asymptotic stability of hybrid stochastic systems with neutral term and non-differentiable time delay, \emph{Syst. Control Lett.} 178 (2023) 105586.

\bibitem{Mao2006}
X. Mao, C. Yuan, \emph{Stochastic Differential Equations with Markovian Switching}, Imperial College Press, London, 2006.

\bibitem{Mao2007}
X. Mao, \emph{Stochastic Differential Equations and Applications}, Elsevier, Cambridge, 2007.

\bibitem{M2013}
M. Milo\u{s}evi\'{c}, Almost sure exponential stability of solutions to highly nonlinear neutral stochastic differential equations with time-dependent delay and the Euler-Maruyama approximation, \emph{Math. Comput. Model.} 57 (3-4) (2013) 887--899.

\bibitem{M2014}
M. Milo\u{s}evi\'{c}, Implicit numerical methods for highly nonlinear neutral stochastic differential equations with time-dependent delay, \emph{Appl. Math. Comput.} 244 (2014) 741--760.

\bibitem{Mo2017c}
H. Mo, X. Zhao, F. Deng, Exponential mean-square stability of the $\theta$-method for neutral stochastic delay differential equations with jumps, \emph{Int. J. Syst. Sci.} 48 (3) (2017) 462--470.

\bibitem{Mo2021}
H. Mo, L. Liu, M. Xing, F. Deng, B. Zhang, Exponential stability of implicit numerical solution for nonlinear neutral stochastic differential equations with time-varying delay and Poisson jumps, \emph{Math. Meth. Appl. Sci.} 44 (7) (2021) 5574--5592.

\bibitem{Mo2017b}
H. Mo, X. Zhao, F. Deng, Mean-square stability of the backward Euler-Maruyama method for neutral stochastic delay differential equations with jumps, \emph{Math. Meth. Appl. Sci.} 40 (5) (2017) 1794--1803.

\bibitem{Ngoc2021}
P.H.A. Ngoc, New results on exponential stability in mean square of neutral stochastic equations with delays, \emph{Int. J. Control} 95 (11) (2022) 3030--3036.

\bibitem{Ky3}
P.H.A. Ngoc, K. Tran, On stability of numerical solutions of neutral stochastic delay differential equations with time-dependent delay, \emph{Math. Meth. Appl. Sci.} 46 (9) (2023) 11246--11261.

\bibitem{Ky2}
P.H.A. Ngoc, T.N.B. Le, K. Tran, Exponential stability in mean square of theta approximations for neutral stochastic delay differential equations with Poisson jumps, \emph{Int. J. Syst. Sci.} (2025) (in press).

\bibitem{N2003}
S.I. Niculescu, \emph{Delay Effects on Stability: A Robust Control Approach}, vol. 269, Springer, 2003.

\bibitem{M2017a}
M. Obradovi\'{c}, M. Milo\u{s}evi\'{c}, Almost sure exponential stability of the $\theta$-Euler-Maruyama method for neutral stochastic differential equations with time-dependent delay when $\theta\in [0,1/2]$, \emph{Filomat} 31 (18) (2017) 5629--5645.

\bibitem{M2017S}
M. Obradovi\'{c}, M. Milo\u{s}evi\'{c}, Stability of a class of neutral stochastic differential equations with unbounded delay and Markovian switching and the Euler-Maruyama method, \emph{J. Comput. Appl. Math.} 309 (2017) 244--266.

\bibitem{Petrovic2021}
A. Petrovi\'{c}, M. Milo\u{s}evi\'{c}, The truncated Euler-Maruyama method for highly nonlinear neutral stochastic differential equations with time-dependent delay, \emph{Filomat} 35 (7) (2021) 2457--2484.

\bibitem{Ky1}
K. Tran, G. Yin, Impulsive stochastic functional differential equations with Markovian switching: study of exponential stability from a numerical solution point of view, \emph{IMA J. Math. Control Info.} 42 (2025) 172--185.

\bibitem{Mao2021}
A. Wu, S. You, W. Mao, X. Mao, L. Hu, On exponential stability of hybrid neutral stochastic differential delay equations with different structures, \emph{Nonlinear Anal. Hybrid Syst.} 39 (2021) 100971.

\bibitem{Mao2008}
F. Wu, X. Mao, Numerical solutions of neutral stochastic functional differential equations, \emph{SIAM J. Numer. Anal.} 46 (4) (2008) 1821--1841.

\bibitem{Zhu}
G. Yin, C. Zhu, \emph{Hybrid Switching Diffusions: Properties and Applications}, Springer, New York, 2010.

\bibitem{Zhou2015}
S. Zhou, Exponential stability of numerical solution to neutral stochastic functional differential equation, \emph{Appl. Math. Comput.} 266 (2015) 441--461.

\bibitem{Zong2014}
Z. Zong, F. Wu, G. Yin, Z. Jin, Almost sure and pth-moment stability and stabilization of regime-switching jump diffusion systems, \emph{SIAM J. Control Optim.} 52 (4) (2014) 2595--2622.

\bibitem{Zong2015}
Z. Zong, F. Wu, C. Huang, Exponential mean square stability of the theta approximations for neutral stochastic differential delay equations, \emph{J. Comput. Appl. Math.} 286 (2015) 172--185.

\bibitem{Zong2016}
Z. Zong, F. Wu, Exponential stability of the exact and numerical solutions for neutral stochastic delay differential equations, \emph{Appl. Math. Model.} 40 (1) (2016) 19--30.

\bibitem{Zhang2019}
Y. Zhang, R. Li, X. Huo, Switching-dominated stability of numerical solutions for hybrid neutral stochastic differential delay equations, \emph{Nonlinear Anal. Hybrid Syst.} 33 (2019) 76--92.

\bibitem{Zhou2013}
S. Zhou, Z. Fang, Numerical approximation of nonlinear neutral stochastic functional differential equations, \emph{J. Appl. Math. Comput.} 41 (1-2) (2013) 427--445.

}
\end{thebibliography}
\end{document}